\documentclass[a4paper]{article}

\usepackage[utf8]{inputenc}
\usepackage{lmodern}
\usepackage[T1]{fontenc}

\usepackage{mathtools, amssymb, amsthm}
\usepackage{enumitem}
\usepackage{bbm,dsfont}
\numberwithin{equation}{section}

\newcommand{\R}{\mathds{R}}
\newcommand{\C}{\mathds{C}}

\newcommand{\p}{\partial}
\newcommand{\la}{\lambda}
\newcommand{\Ng}{\mathcal{N}_{\Gamma_1}}

\renewcommand{\Re}{\operatorname{Re}}

\newcommand{\fra}{\mathfrak{a}}
\newcommand{\frb}{\mathfrak{b}}

\theoremstyle{plain}
\newtheorem{theorem}{Theorem}[section]

\newtheorem{lemma}[theorem]{Lemma}
\newtheorem{corollary}[theorem]{Corollary}

\theoremstyle{definition}

\theoremstyle{remark}
\newtheorem{remark}[theorem]{Remark}
\begin{document}

\title{A "milder" version of Calder\'on's inverse problem for anisotropic conductivities  and partial data}
\author{
    El Maati Ouhabaz}


\date{}
\maketitle

\begin{abstract}\label{abstract}
Given  a general symmetric elliptic operator 
$$ L_{a} := \sum_{k,j=1}^d \p_k (a_{kj} \p_j ) + \sum_{k=1}^d a_k \p_k - \p_k( \overline{a_k} . ) + a_0$$
we define the associated Dirichlet-to-Neumann (D-t-N) map  with  partial data, i.e.,  data supported in a part of the boundary. We prove positivity, $L^p$-estimates  and domination properties for the semigroup associated with this D-t-N operator. Given $L_a $ and $L_b$ of the previous type with bounded measurable coefficients $a = \{ a_{kj}, \ a_k, a_0 \}$ and 
$b = \{ b_{kj}, \ b_k, b_0 \}$, we prove that if  their  partial D-t-N operators (with $a_0$ and $b_0$ replaced by $a_0 -\la$ 
and $b_0 -\la$)  coincide for all $\la$,   then the operators $L_a$ and $L_b$, endowed 
with Dirichlet, mixed or Robin  boundary conditions are unitarily equivalent. In  the case of the Dirichlet boundary conditions, 
this result was  
 proved recently by Behrndt and Rohleder \cite{BR12}.   We provide a different 
proof, based on spectral theory, which  works for  other boundary conditions. 
 \end{abstract}
 

\section{Introduction}\label{section:introduction}
Let $\Omega$ be a bounded Lipschitz domain of $\R^d$ with boundary $\p\Omega$. Let $\Gamma_0$ be a closed subset of 
$\p\Omega$ with  $\Gamma_0 \not= \p\Omega$  and $\Gamma_1$ its complement in $\p\Omega$.  We consider the symmetric elliptic operator
on $L^2(\Omega)$ given by the formal expression:
$$ L_{a}(\lambda) := \sum_{k,j=1}^d \p_k (a_{kj} \p_j ) + \sum_{k=1}^d a_k \p_k - \p_k( \overline{a_k} . ) + a_0 - \lambda$$
where $a_{kj} = \overline{a_{jk}}, a_k, a_0 = \overline{a_0}  \in L^\infty(\Omega)$ and $\lambda$ is a constant. 
We define the associated Dirichlet-to-Neumann (D-t-N) operator, $\mathcal{N}_{\Gamma_1, a}(\lambda)$, with partial data as follows:\\
 for  $\varphi \in H^{\frac{1}{2}}(\p\Omega)$ with $\varphi = 0 $ on  $\Gamma_0$, one solves the Dirichlet problem 
\begin{equation}\label{eqh1}
 L_a(\lambda) u = 0 \, \, \text{weakly in}\  W^{1,2}(\Omega)\, \text{with}\ u = \varphi  \, \,   \text{on}\, \p\Omega,
 \end{equation}
and defines  (in the weak sense)
\begin{equation}\label{DTN1}
{\mathcal N}_{\Gamma_1, a}(\lambda)   \varphi := \sum_{j=1}^d \left( \sum_{k=1}^d a_{kj} \p_k u + \overline{a_j} \varphi \right)\!\nu_j \, \,   \text{on}\,  \Gamma_1.
\end{equation}
Here $\nu = (\nu_1, \cdots, \nu_d)$ is the outer unit normal to the boundary of $\Omega$. The operator 
$\mathcal{N}_{\Gamma_1, a}(\lambda)$ is interpreted as the conormal derivative on the boundary. It is an operator acting on 
$L^2(\p\Omega)$. See Section \ref{sec1} for more details. 

Let us consider first the case where $a_{kj} = \sigma(x) \delta^{kj}$, $ a_k = 0, k= 0,1 \dots d$, where $\sigma \in L^\infty(\Omega)$  is bounded from below (by a positive constant).  
 A. Calder\'on's well known inverse problem asks  whether  one could  determine solely 
 the conductivity $\sigma(x)$   from boundary measurements, i.e., from
$ \Ng(0)$. For the global boundary measurements, i.e., $\Gamma_1 = \p\Omega$, the  first global uniqueness result was proved 
by  Sylvester and Uhlmann \cite{SU87} for a $C^2$-smooth  conductivity 
when  $d \ge 3$.  This results was extended to $C^{1+ \epsilon}$-smooth conductivity by Greenleaf, Lassas and 
Uhlmann \cite{GLU} and then by Haberman and Tataru \cite{HT} to  $C^1$ and  Lipschitz conductivity close to the identity. Haberman \cite{Hab} proved the uniqueness for Lipschitz conductivity when $d= 3, 4$ and this was extended to all $d \ge 3$ by Caro and Rogers \cite{CR}.
In the two-dimension case with $C^2$-smooth conductivity, the global uniqueness
was proved by Nachman \cite{Na96}. This regularity assumption was completely removed by Astala and P\"aiv\"arinta \cite{AP06}
dealing with $\sigma \in L^\infty(\Omega)$.  \\
The inverse problem with partial data consists in proving uniqueness (either for the isotropic conductivity or for the potential) when the measurement is made only on a part of the boundary. This means that the trace of the solution $u$ in  \eqref{eqh1} is supported on a set $\Gamma_D$ and the D-t-N operator  is known on  $\Gamma_N$ for some parts $\Gamma_D$ and $\Gamma_N$ of the boundary. 
This problem has been studied  and there are some  geometric conditions on  $\Gamma_D$ and $ \Gamma_N$ under  which uniqueness is proved. We refer to Isakov \cite{Isa},  Kenig, Sj\"ostrand and Uhlmann \cite{KSU},  Dos Santos et al. \cite{DKSU}, 
Imanuvilov, Uhlmann and Yamamoto \cite{IUY} and the review paper \cite{KS} by Kenig and Salo  for more  references and recent developments. 

Now we move to the anisotropic case. This corresponds to the general case  where the conductivity  is given by a general matrix $a_{kj}$. 
As pointed out by  Lee and Uhlmann in \cite{LU89}, it is not difficult to see that a change of variables given by a diffeomorphism of 
$\Omega$ which is the identity on $\p\Omega$  leads to different coefficients $b_{kj}$ without changing the D-t-N operator on the boundary. Therefore 
the single coefficients $a_{kj}$ are  not uniquely determined in general.  
In \cite{LU89}, Lee and Uhlmann   proved  that for real-analytic coefficients  the uniqueness up to  a diffeomorphism  
holds when  the dimension $d$ is $ \ge 3$. The same result  was proved by Astala, Lassas and P\"aiv\"arinta \cite{ALP} for the case $d= 2$ and $L^\infty$-coefficients. 

In \cite{BR12}, Behrndt and Rohleder  considered  general elliptic expressions  $L_a$ and $L_b$ as above and prove that
if the corresponding D-t-N operators ${\mathcal{N}_{\Gamma_1, a}}(\lambda)$ and ${\mathcal{N}_{\Gamma_1, b}}(\lambda)$
coincide for all $\la$ in a set having an accumulation point in $\rho(L^D_a) \cap \rho(L^D_b)$ then the operators 
$L_a^D$ and $L_b^D$  are unitarily equivalent. Here $L_a^D$ is the elliptic operator $L_a$ with Dirichlet boundary conditions.
This can be seen as a milder version of the uniqueness problem discussed above. The proof is based on the theory of extensions of symmetric operators and unique continuation results. It is assumed in \cite{BR12} that the coefficients are Lipschitz continuous on 
$\overline{\Omega}$. 
We give a different proof of this result which also works for   other boundary conditions.  
Our main result is the following. 
\begin{theorem}\label{thm0} Suppose that $\Omega$ is a bounded Lipchitz domain of $\R^d$ with $d \ge 2$. Let $\Gamma_0$ be a closed  subset of $\p\Omega$, $\Gamma_0 \not= \p\Omega$  and $\Gamma_1$ its complement. 
Let $a= \{ a_{kj}, a_k, a_0 \}$ and $b = \{ b_{kj}, b_k, b_0 \}$ be bounded functions on $\Omega$ such that $a_{kj}$ and $b_{kj}$ satisfy the usual ellipticity condition.  If $d \ge 3$ we assume in addition that  the coefficients $a_{kj}, b_{kj}, a_k$ and $b_k$ are Lipschitz continuous on  $\overline{\Omega}$. \\
Suppose that ${\mathcal{N}_{\Gamma_1, a}}(\lambda)  = {\mathcal{N}_{\Gamma_1, b}}(\lambda)$  for all  $\lambda $ in a set having an accumulation point in $\rho(L^D_a) \cap \rho(L^D_b)$. Then:\\
$i)$ The operators $L_a$ and $L_b$ endowed with Robin  boundary conditions are unitarily equivalent.\\
$ii)$  The operators $L_a$ and $L_b$ endowed with mixed  boundary conditions (Dirichlet  on $\Gamma_0$ and Neumann type on $\Gamma_1$) are unitarily equivalent.\\
$iii)$ The operators $L_a$ and $L_b$  endowed with Dirichlet boundary conditions are unitarily equivalent.

In addition, for Robin or mixed boundary conditions, the eigenfunctions associated to the same eigenvalue  $\lambda \notin \sigma(L_a^D) = \sigma(L_b^D)$ coincide  on the boundary of $\Omega$.
\end{theorem}

Note that unlike \cite{BR12} we do not assume regularity of the coefficients when $d= 2$.

We shall restate this  theorem in a more precise way after introducing  some  necessary material and notation. 
The proof is given  in Section \ref{sec3}.  It  is based on spectral theory and differs  from the one in  \cite{BR12}.   Our   strategy is to use a relationship between eigenvalues of the D-t-N operator ${\mathcal{N}_{\Gamma_1, a}}(\lambda)$  and eigenvalues of the elliptic operator with Robin boundary conditions
$L_a^\mu$ on $\Omega$ where $\mu$ is a parameter.   One of the main ingredients in  the proof is  that each eigenvalue 
of the latter operator is a strictly decreasing map with respect to the parameter $\mu$. Next, the equality of ${\mathcal{N}_{\Gamma_1, a}}(\lambda)$ and 
${\mathcal{N}_{\Gamma_1, b}}(\lambda)$ allows us  to prove that the spectra of $L_a^\mu$ and $L_b^\mu$ are the same and the eigenvalues have the same multiplicity.  
The similarity of the two elliptic operators with Dirichlet boundary conditions
is obtained from the similarity of  $L_a^\mu$ and $L_b^\mu$ by letting the parameter $\mu$ tend to $-\infty$. During the  proof we use  some ideas  from the papers of Arendt and Mazzeo \cite{AM07} and \cite{AM12}  which deal with a different subject, namely the Friendlander inequality for the eigenvalues of the Dirichlet and Neumann Laplacian on a Lipschitz domain. The ideas  which we borrow  from \cite{AM07} and \cite{AM12} are then adapted and extended to   our general case of D-t-N operators with variable coefficients and partial data.\\
In Section \ref{sec1} we define the D-t-N operator with partial data using the method of sesquilinear  forms. In particular, for symmetric coefficients it is a self-adjoint operator on $L^2(\Gamma_1)$. It can be seen as an operator on 
$L^2(\p\Omega)$ with a non-dense domain and which we extend  by $0$ 
to  $L^2(\Gamma_0)$. Therefore one can associate   with this  D-t-N operator
a semigroup $(T^{\Gamma_1}_t)_{t\ge 0}$  acting   on $L^2(\p\Omega)$. In Section \ref{sec2} we prove positivity,  sub-Markovian  and domination properties for such semigroups.  In particular,  $(T^{\Gamma_1}_t)_{t\ge 0}$  extends to  a  contraction semigroup on $L^p(\p\Omega)$ for all $p \in [1, \infty)$. Hence, for $\varphi_0 \in L^p(\Gamma_1)$, one obtains existence and uniqueness of the solution in $L^p(\p\Omega)$ to the evolution problem
$$\partial_t \varphi  + {\mathcal{N}_{\Gamma_1, a}}(\lambda) \varphi = 0, \quad \varphi (0) = \varphi_0.$$
The results of Section \ref{sec2} are of independent interest and are not used in the proof of the  theorem stated above. 
\section{The partial D-t-N operator}\label{sec1} 

Let $\Omega$ be a bounded  open set  of $\R^d$ with Lipschitz boundary $\p \Omega$. The boundary is endowed with the $(d-1)$-dimensional Hausdorff measure $d\sigma$.  Let 
$$a_{kj}, a_k, \tilde{a_k}, a_0: \Omega \to \C$$
be bounded measurable  for $1 \le k, j \le d$ and such that there exists a constant $\eta > 0$ for which 
\begin{equation}\label{1-1}
 \Re \sum_{k,j=1}^d a_{kj} (x) \xi_k \overline{\xi_j} \ge \eta | \xi |^2
\end{equation} 
for all  $\xi = (\xi_1, \cdots, \xi_d) \in \C^d$ and a.e. $x \in \Omega$.  \\
Let $\Gamma_0$ be an closed  subset of $\p\Omega$ and $\Gamma_1$ its complement in $\p \Omega$. \\

\noindent\underline{{\it Elliptic operators on $\Omega$.}}\\

We consider  the space 
\begin{equation}\label{1-3}
V = \{ u \in W^{1,2}(\Omega),\,  \text{Tr}(u) = 0\  \text{on}\   \Gamma_0 = 0 \}, 
\end{equation}
where $\text{Tr}$ denotes the trace operator. We define the sesquilinear form 
$$\fra : V \times V \to \C$$
 by the expression 
\begin{equation}\label{1-4}
\fra(u,v) = \sum_{k,j=1}^d \int_\Omega a_{kj} \p_k u \overline{\p_j v}\ dx + \sum_{k=1}^d \int_\Omega a_k \p_k u \overline{ v} 
+ \tilde{a_k} u  \overline{ \p_k v}\ dx + a_0 u \overline{v}\ dx
\end{equation}
for all $u, v \in V$. Here we use the notation $\p_j $ for the partial derivative $\frac{\p}{\p x_j}$. 

It follows easily from the ellipticity assumption (\ref{1-1}) that the form $\fra$ is quasi-accretive, i.e., there exists a constant $w$ such that 
\[ \Re \fra(u,u) + w \| u \|_2^2 \ge 0 \ \ \forall u \in V.
\]
In addition, since $V$ is a closed subspace of $W^{1,2}(\Omega)$ the form $\fra$ is closed. Therefore there exists an  operator 
$L_a$ associated with $\fra$. It is defined by 
\begin{align*}
D(L_a) &= \{ u \in V, \exists v \in L^2(\Omega): \fra(u, \phi) = \int_\Omega v \overline{\phi}\ dx \,\,\,  \forall \phi \in V \},\\
L_a u &:=  v.
\end{align*}
Formally, $L_a$ is given by the expression
\begin{equation}\label{1-5}
L_a u = - \sum_{k, j=1}^d \p_k (a_{kj} \p_j u ) + \sum_{k=1}^d a_k \p_k u  - \p_k( \tilde{a_k} u ) + a_0 u.
\end{equation}
In addition, following \cite{AM07} or  \cite{AM12} we define the conormal derivative $\frac{\p}{\p \nu}$ in the weak sense (i.e. in $H^{-1/2}(\p\Omega)$ the dual space of $H^{1/2}(\p\Omega) = \text{Tr}(W^{1,2}(\Omega))$), then 
$L_a$ is subject to the boundary conditions
\begin{equation}\label{1-6}
\left\{ 
\begin{aligned} \text{Tr}(u) &= 0 \quad \text{on } \Gamma_0\\
 \frac{\p u}{\p \nu} &= 0 \quad \text{on } \Gamma_1. 
 \end{aligned}\right.
\end{equation}
The conormal derivative in our case  is usually  interpreted as
$$\sum_{j=1}^d \left( \sum_{k=1}^d a_{kj} \p_k u  + \tilde{a_j} u\right) \nu_j,$$
 where 
$\nu = (\nu_1, \cdots, \nu_d)$ is the outer unit normal to the boundary of $\Omega$.  For all this see \cite{Ouh05}, Chapter 4.\\
The condition (\ref{1-6})  is a  mixed boundary condition which consists in taking  
Dirichlet  on $\Gamma_0$ and Neumann type boundary condition on $\Gamma_1$. For this reason we denote this
 operator by  $L^M_a$. 
The subscript $a$  refers to the fact that the coefficients of the operator are given by $a = \{ a_{kj}, a_k, \tilde{a_k},  a_0 \}$ and $M$ refers to  mixed boundary conditions.  

We  also define the elliptic operator  with Dirichlet boundary condition $\text{Tr}(u) = 0 $ on $\p\Omega$. It is the operator associated with the form given by the expression 
\eqref{1-4}  with domain $D(\fra) = W^{1,2}_0(\Omega)$. It is a quasi-accretive and closed form and its associated operator $L^D_a$ has the same expression as in 
\eqref{1-5} and subject to  the Dirichlet boundary condition $\text{Tr}(u) = 0$ on $\p\Omega$. 

Similarly, we define $L_a^N$ to be  the elliptic operator with Neumann type boundary conditions
$$ \frac{\p u}{\p \nu} = 0 \quad \text{on } \p\Omega.$$
It is the operator associated with the form given by the expression 
\eqref{1-4} with  domain $D(\fra) = W^{1,2}(\Omega)$.\\
Note that $L_a^D$ coincides with   $L_a^M$ if  $\Gamma_0 = \p\Omega$ and $L_a^N$ coincides with   $L_a^M$ if  
 $\Gamma_0 = \emptyset$. \\
 Finally we define  elliptic operators with Robin boundary conditions. Let $\mu \in \R$ be a  constant and define 
 \begin{align}\label{Rob}
\fra^\mu(u,v) &= \sum_{k,j=1}^d \int_\Omega a_{kj} \p_k u \overline{\p_j v}\ dx + \sum_{k=1}^d \int_\Omega a_k \p_k u \overline{ v} 
+ \tilde{a_k} u  \overline{ \p_k v}\ dx + a_0 u \overline{v}\ dx \nonumber\\
&{} \hspace{.3cm}   - \mu  \int_{\p\Omega} \text{Tr}(u) \overline{\text{Tr}(v)} d\sigma
\end{align}
for all $u, v \in D(\fra^\mu) :=  V$. Again,  $\text{Tr}$ denotes the trace operator. 
Using the standard inequality (see \cite{AM07} or  \cite{AM12}), 
$$ \int_{\p\Omega} | \text{Tr}(u) |^2 \le \varepsilon \| u \|_{W^{1,2}(\Omega)}^2 + c_\varepsilon \int_{\Omega} | u |^2$$
which is valid for all $\varepsilon > 0$ ($c_\varepsilon$ is a constant depending on $\varepsilon$) one obtains that for some positive constants $w$ and $\delta$
$$\Re \fra^\mu (u,u) + w  \int_\Omega | u |^2 \ge \delta  \| u \|_{W^{1,2}(\Omega)}^2.$$
From this it follows that $\fra^\mu$ is a quasi-accretive and closed sesquilinear form. 
One can associate with $\fra^\mu$ an operator $L_a^\mu$. This operator has the same expression \eqref{1-5} and it is subject to the Robin boundary conditions
\begin{equation}\label{3-6}
\left\{ 
\begin{aligned} \text{Tr}(u) &= 0 \quad \text{on } \Gamma_0\\
  \frac{\p u}{\p \nu} &=  \mu\ \text{Tr}(u)  \quad \text{on } \Gamma_1. 
 \end{aligned}\right.
\end{equation}
Actually, the boundary conditions (\ref{3-6}) are mixed Robin boundary conditions  in the sense that we have the Dirichlet condition on $\Gamma_0$ and the Robin one on $\Gamma_1$. For simplicity we ignore the word "mixed" and refer to (\ref{3-6}) as the  Robin boundary conditions.\\
According to our previous notation, if $\mu = 0$, then $\fra^0 = \fra$ and $L_a^0 = L_a^M$. \\
Note  that we may choose here  $\mu $ to be a bounded  measurable function on the boundary rather than just a constant. \\
 

\noindent\underline{{\it The partial Dirichlet-to-Neumann operator on $\p\Omega$.}}\\

Suppose as before that $a = \{ a_{kj}, a_k, \tilde{a_k},  a_0 \}$ are bounded measurable and  satisfy the ellipticity condition
\eqref{1-1}.  Let $\Gamma_0, \Gamma_1, V$ be as above and  $\fra$ is the sesquilinear form  defined by \eqref{1-4}.\\
We define the space
\begin{equation}\label{VH}
V_H := \{ u \in V, \fra(u, g) = 0 \,\, \text{for all } g \in W^{1,2}_0(\Omega) \}. 
\end{equation}
Then $V_H$ is a closed subspace of $V$. It is interpreted as  the space of harmonic functions for  the operator $L_a$ (given by 
\eqref{1-5}) with the  additional property that $\text{Tr}(u) = 0$ on $\Gamma_0$.\\
 We start with  the following simple lemma.
\begin{lemma}\label{lem1}
 Suppose that  $0 \notin \sigma(L^D_a)$.  Then  
 \begin{equation}\label{1-10}
 V = V_H \oplus W^{1,2}_0(\Omega).
 \end{equation}
 \end{lemma}
 \begin{proof} We argue as in \cite{EO13}, Section 2 or \cite{AM07}. Let us denote by $\fra^D$ the 
  form associated with $L_a^D$, that is,  
$ \fra^D$ is given by \eqref{1-4} with  $D(\fra^D) = W^{1,2}_0(\Omega)$. There exists an operator 
${\mathcal L}_a^D: W^{1,2}_0(\Omega) \to W^{-1,2}(\Omega) := W^{1,2}_0(\Omega)' $ (the anti-dual of $W^{1,2}_0(\Omega)$) associated with $\fra^D$ in the sense
$$ \langle {\mathcal L}_a^D h, g \rangle = \fra^D (h, g)$$
for all $h, g \in W^{1,2}_0(\Omega)$. The notation $\langle \cdot , \cdot \rangle$ denotes the duality $W^{1,2}_0(\Omega)'- W^{1,2}_0(\Omega)$.
 Since $0 \notin \sigma(L_a^D)$, then $L_a^D$ is invertible. Therefore ${\mathcal L}_a^D$, seen  as operator on $W^{1,2}_0(\Omega)'$ with domain
 $W^{1,2}_0(\Omega)$, is also invertible on $W^{1,2}_0(\Omega)'$ since the two operators $L_a^D$ and ${\mathcal L}_a^D$ 
 have the same spectrum (see e.g., \cite{ABHN}, Proposition 3.10.3).  Now we fix $u \in V$ and consider the (anti-)linear functional
 $$ F: v \mapsto  \fra(u, v).$$
 Clearly, $F \in W^{1,2}_0(\Omega)'$ and hence there exists a unique $u_0 \in W^{1,2}_0(\Omega)$ such that ${\mathcal L}_a^Du_0 = F$, i.e.,
  $\langle {\mathcal L}_a^D u_0, g \rangle = F(g)$ for all $g \in W^{1,2}_0(\Omega)$. 
  This means that 
 $\fra(u-u_0, g) = 0$ for all $g \in W^{1,2}_0(\Omega)$ and hence $u-u_0 \in V_H$. Thus,  $u = u - u_0 + u_0 \in V_H + W^{1,2}_0(\Omega).$ 
 Finally, if $u \in V_H \cap W^{1,2}_0(\Omega)$ then 
 $\fra(u, g) = 0$ for all $g \in W^{1,2}_0(\Omega)$. This means that $u \in D(L_a^D)$ with $L_a^D u = 0$. Since $L_a^D$ is invertible we conclude that $ u = 0$. 
  \end{proof}
 
As a consequence of Lemma \ref{lem1}, the trace operator  $\text{Tr}: V_H \to L^2(\p\Omega)$ is injective and  
\begin{equation}\label{1-101}
\text{Tr} (V_H)  = \text{Tr} (V).
\end{equation}

In the rest of this section we assume  that $0 \notin \sigma(L^D_a)$.  We  define on $L^2(\p\Omega, d\sigma)$ the sesquilinear form
\begin{equation}\label{1-7}
\frb(\varphi, \psi) := \fra(u,v)
\end{equation}
where $u, v \in V_H$ are such that $ \varphi = \text{Tr}(u)$ and  $ \psi = \text{Tr}(v).$
 This means that  $D(\frb) = \text{Tr}(V_H)$ and by \eqref{1-101}
 \begin{equation}\label{dom}
 D(\frb) = \text{Tr} (V_H)  = \text{Tr} (V).
 \end{equation}
 
\begin{lemma}\label{lem2}
There exist positive constants $w$, $\delta$ and $M$ such that
\begin{equation}\label{1111}
  \Re \frb(\varphi, \varphi) + w \int_{\p\Omega} | \varphi|^2 \ge \delta \| u \|_{W^{1,2}(\Omega)}^2
 \end{equation}
and 
\begin{equation}\label{1112}
| \frb(\varphi, \psi) | \le M \left[  \Re \frb(\varphi, \varphi) +  w \int_{\p\Omega} | \varphi|^2\right]^{1/2} \left[ \Re  \frb(\psi, \psi) + 
 w \int_{\p\Omega} | \psi |^2\right]^{1/2}
 \end{equation}
 for all $\varphi, \psi \in D(\frb)$. In the first inequality, $u \in V_H$ is such that $\text{Tr}(u) = \varphi$.
 \end{lemma}
 \begin{proof} It is well known that $\text{Tr} : W^{1,2}(\Omega) \to L^2(\p\Omega)$ is a compact operator and 
 since $\text{Tr}: V_H \to L^2(\p\Omega)$ is injective  it follows that  for every $\epsilon > 0$ there exists a constant $c > 0$ 
 such that 
 \begin{equation}\label{1-104}
 \int_\Omega | u |^2 \le \epsilon \| u \|_{W^{1,2}}^2 + c \int_{\p\Omega} | \text{Tr}(u)|^2
 \end{equation}
 for all $u \in V_H$ (see, e.g., \cite{AM07}).  In particular,
 \begin{equation}\label{1-1041}\int_\Omega | u |^2 \le \frac{\epsilon}{1- \epsilon} \int_\Omega |\nabla u|^2  + \frac{c}{1-\epsilon}\int_{\p\Omega} | \varphi |^2.
 \end{equation}
 Now, let $\varphi \in D(\frb) = \text{Tr}(V_H)$ and $u \in V_H$  such that $\varphi = \text{Tr}(u)$. It follows from the ellipticity  assumption 
 \eqref{1-1} and the boundedness of the coefficients that for some constant $c_0 > 0$
 \begin{equation*}
 \Re \fra(u, u)\ge  \frac{\eta}{2} \int_\Omega | \nabla\ u |^2 - c_0 \int_\Omega | u |^2.
 \end{equation*}
 Therefore, using  \eqref{1-1041} and the definition of $\frb$ we obtain
 \begin{eqnarray*}
  \Re \frb(\varphi, \varphi) &=& \Re \fra(u,u) \\
  &\ge&   (\frac{\eta}{2}  - \frac{c_0 \epsilon}{1-\epsilon}) \int_\Omega | \nabla\ u |^2 - \frac{c c_0}{1-\epsilon} \int_{\p\Omega} | \varphi |^2.
 \end{eqnarray*}
  Taking $\epsilon > 0$ small enough  we obtain  \eqref{1111}.\\
 In order to prove the second inequality, we use the definition of $\frb$ and again the boundedness of the coefficients to see that
 \begin{eqnarray*}
 | \frb(\varphi, \psi)  | &=& | \fra(u,v)|\\
 &\le& C \| u \|_{W^{1,2}} \| v \|_{W^{1,2}}.
 \end{eqnarray*}
 Thus,  \eqref{1112} follows from \eqref{1111}.
 \end{proof}
\begin{corollary}\label{pro1}
The form $\frb$ is continuous, quasi-accretive and closed. 
\end{corollary}
\begin{proof}
Continuity of $\frb$ is exactly \eqref{1112}.  Quasi-accretivity means that 
$$\Re \frb(\varphi, \varphi) + w \int_{\p\Omega} | \varphi|^2 \ge 0$$
for some $w$ and all $\varphi \in D(\frb)$. This  follows from \eqref{1111}. \\
Now we prove that $\frb$ is closed which means that $D(\frb)$ is complete for the norm
$$ \| \varphi \|_\frb := \left(\Re \frb(\varphi, \varphi) + w \int_{\p\Omega} | \varphi|^2 \right)^{1/2}$$
in which  $w$ is as in  \eqref{1111}. If $(\varphi_n)$ is a Cauchy sequence for $\| \cdot \|_\frb$ then by 
\eqref{1111} the corresponding $(u_n) \in V_H$ with $\text{Tr}(u_n) = \varphi_n$ is a Cauchy sequence in $V_H$. 
Since  $V_H$ is a closed subspace of  $V$ it follows that $u_n$ is convergent to some $u$ in $V_H$. Set $\varphi := \text{Tr}(u)$.  We have $\varphi \in D(\frb)$ and the definition of $\frb$ together with continuity of $\text{Tr}$ as an operator from 
$W^{1,2}(\Omega)$ to $L^2(\p\Omega)$ show that $\varphi_n$ converges to 
$\varphi$ for the norm $\| \cdot \|_\frb$. This means that $\frb$ is a closed form. 
\end{proof}

Note that the domain $\text{Tr}(V_H)$ of $\frb$ may not be dense in $L^2(\p\Omega)$ since functions in this domain vanish on
 $\Gamma_0$. Indeed,
 \begin{equation}\label{eq01010}
  H := \overline{D(\frb)}^{L^2(\p\Omega)} = L^2(\Gamma_1) \oplus \{0\}.
  \end{equation}
 The direct inclusion follows from the fact that if $\varphi_n \in D(\frb)$ converges in $L^2(\p\Omega)$ then after extracting a subsequence
 we have a.e. convergence. Since $\varphi_n = 0$ on $\Gamma_0$ we obtain that the limit $\varphi = 0$ on $\Gamma_0$. The reverse inclusion can  be proved as follows. Let $\Gamma_2$ be a closed subset of $\R^d$ with  $\Gamma_2 \subset \Gamma_1$ and consider the space
 $E = \{ u_{\vert \Gamma_2}: u \in W^{1, \infty}(\R^d),  u_{\vert \Gamma_0} = 0 \}$. Then $E \subset C(\Gamma_2)$ and an easy application of the Stone-Weierstrass theorem shows that $E$ is dense in $C(\Gamma_2)$.  Now given $\varphi \in C_c(\Gamma_1)$ and $\epsilon > 0$ we find 
 $\Gamma_2 $ such that $\| \mathbbm{1}_{\Gamma_1 \setminus \Gamma_2} \|_2 < \epsilon $ and $u_{\vert \Gamma_2} \in E$ such that 
 $\| u_{\vert \Gamma_2} - \varphi \|_{C(\Gamma_2)} < \epsilon$. Finally we take $\chi \in C_c^\infty(\R^d)$ such that $\chi = 1$ on $\Gamma_2$. Then $(u \chi)_{\vert \Omega} \in V$ and 
 \begin{eqnarray*}
 \| u \chi - \varphi \|_{L^2(\Gamma_1)} &\le& \| u - \varphi \|_{L^2(\Gamma_2)} + \| \chi \|_{L^2(\Gamma_1\setminus \Gamma_2)}\\
 &\le& \epsilon | \Gamma_2 | + \| \chi \|_\infty \epsilon.
 \end{eqnarray*}
 Here $ | \Gamma_2 |$ denotes  the measure of $ \Gamma_2 $.  These inequalities together with the fact that $C_c(\Gamma_1)$ is dense in 
 $L^2(\Gamma_1)$ imply   \eqref{eq01010}. 
 \vspace{.5cm}
 
We return to the form $\frb$ defined above. We  associate  with $\frb$ an operator  $\Ng$. It is  defined by 
$$D(\Ng) := \{ \varphi \in D(\frb), \exists \psi \in H: \frb(\varphi, \xi) = \int_{\Gamma_1} \psi \overline{\xi} \, \; \forall \xi \in D(\frb) \}, \quad \Ng \varphi = \psi.$$
The operator $\Ng$ can be interpreted  as an operator on $L^2(\p\Omega)$ defined as follows:  if $\varphi \in D(\Ng)$ then there exists a unique $u \in V_H$ such that $\varphi = \text{Tr }(u)$ and  
\begin{equation}\label{1-11}
\varphi_{| \Gamma_0} = 0, \quad  \Ng (\varphi) = \frac{\partial u}{\partial \nu} \ \text{on } \Gamma_1.
\end{equation}
Again $\frac{\partial u}{\partial \nu}$ is interpreted in the weak sense as the conormal derivative  that is $\sum_{j=1}^d \left( \sum_{k=1}^d a_{kj} \p_k u + \tilde{a_j} \varphi \right)\!\nu_j$. 
 In the particular case where $a_{kj} = \delta_{kj}$ and $a_1= \cdots = a_d = 0$ the right hand side of (\ref{1-11}) is seen as  
 the normal derivative on the boundary. All this  can be made precise by applying the Green formula if the boundary and the coefficients are smooth enough.
 
\noindent We  call  $\Ng$ the {\it partial Dirichlet-to-Neumann} operator on $L^2(\p\Omega)$ or the {\it Dirichlet-to-Neumann operator with partial data}. The term {\it partial} refers to the fact that $\Ng$ is known only on the part $\Gamma_1$ of the boundary $\p\Omega$. \\
It follows from the  general theory of forms that  $-\Ng$ generates a holomorphic semigroup $e^{-t\Ng}$ on $H$. We define $T_t^{\Gamma_1}$ on $L^2(\p\Omega)$ by
$$ T_t^{\Gamma_1} \varphi = e^{-t \Ng} (\varphi \mathbbm{1}_{\Gamma_1}) \oplus 0.$$
We shall refer to $(T_t^{\Gamma_1})_{t\ge0}$ as the "semigroup" generated by $-\Ng$ on $L^2(\p\Omega)$. It is clear that 
\begin{equation}\label{2-1}
\| T^{\Gamma_1}_t \|_{{\mathcal L}(L^2(\p\Omega))} \le  e^{-w_0 t}, \ \ t \ge 0,
\end{equation}
for some constant $w_0$. Note that if the form $\fra$ is symmetric, then $\frb$ is also symmetric and hence $\Ng$ is self-adjoint. In this case, \eqref{2-1} holds with  $w_0 = \inf \sigma(\Ng)$ which also coincides with  the first eigenvalue of $\Ng$.  For all this, see e.g.  \cite{Ouh05}, Chapter 1.
\section{Positivity and domination}\label{sec2}
In this section we study some  properties of the semigroup $(T^{\Gamma_1}_t)_{t\ge 0}$. We assume throughout this section  that
\begin{equation}\label{sym}
a_{jk} = a_{kj}, \, \tilde{a_k} = a_k, a_0  \in L^\infty(\Omega, \R).
\end{equation}
 We recall that  $L^D_a$ is the elliptic  operator with Dirichlet boundary conditions defined in the previous section.  Its associated symmetric form  $\fra^D$ is given by \eqref{1-4} and has domain $W^{1,2}_0(\Omega)$.  We shall need the accretivity assumption of 
$\fra^D$ (or equivalently  the self-adjoint operator $L^D_a$ is non-negative) which means that 
\begin{equation}\label{2acc}
 \fra^D(u,u) \ge 0 \, \, \text{for all } u \in W^{1,2}_0(\Omega).
\end{equation}
\begin{theorem}\label{thm2-1}
Suppose that $0 \notin \sigma(L^D_a)$,  \eqref{sym} and that $L^D_a$ is accretive. \\
a) The semigroup $(T^{\Gamma_1}_t)_{t\ge 0}$ is positive (i.e., it maps non-negative functions of $L^2(\p\Omega)$ into non-negative functions).\\
b) Suppose in addition that $a_0 \ge 0$  and $a_k = 0$ for all $k \in \{1, \cdots, d\}$. Then $(T^{\Gamma_1}_t)_{t\ge 0}$ is a sub-Markovian semigroup. 
\end{theorem}
Recall that the sub-Markovian property means that for  $\varphi \in L^2(\p\Omega)$ and $t \ge 0$
$$ 0 \le \varphi \le 1 \Rightarrow 0 \le T^{\Gamma_1}_t \varphi \le 1.$$
This property implies in particular that  $(T^{\Gamma_1}_t)_{t\ge 0}$ extends from $L^2(\p\Omega)$ to $L^p(\p\Omega)$  for all 
$p \in [2, \infty[$. Since $\fra$ is symmetric then so is $\frb$ and one obtains  by duality that  $(T^{\Gamma_1}_t)_{t\ge 0}$ extends also to
 $L^p(\p\Omega)$ for $p \in [1, 2]$. 

\begin{proof} The proof  follows  exactly the same lines as for Theorem 2.3 in \cite{EO13}.\\
a) By the well known Beurling--Deny criteria (see \cite{Dav2}, Section~1.3
or \cite{Ouh05}, Theorem 2.6), it suffices to prove that 
$\varphi^+ \in D(\frb)$ and $\frb(\varphi^+, \varphi^-) \leq 0$
for all real-valued $\varphi \in D(\frb)$. Note that the fact that $D(\frb)$ is not densely defined does not affect the the statements of the Beurling-Deny criteria. \\
Let $\varphi \in D(\frb)$ be real-valued. 
There exists a real-valued $u \in H_V$ such that $\varphi = \text{Tr}(u)$.
Then $\varphi^+ = \text{Tr} (u^+) \in \text{Tr}(V) = \text{Tr} H_V = D(\frb)$.  This follows from the fact that 
$v^+ \in V$ for all $v \in V$ (see \cite{Ouh05}, Section 4.2). \\
By Lemma \ref{lem1} we can write $u^+  = u_0 + u_1 $ and $u^-  = v_0 + v_1 $ 
with $u_0, v_0 \in W_0^{1,2}(\Omega)$ and $u_1, v_1 \in H_V$.
Hence, $ u = u^+ - u^- = (u_0 - v_0) + (u_1 - v_1)$.
Since  $u, u_1 -v_1 \in H_V$ it follows that $u_0 = v_0$.
Therefore,
\begin{eqnarray*}
\frb (\varphi^+, \varphi^-) 
& = & \fra(u_1,v_1)
= \fra(u_1, v_0 + v_1)
= \fra(u_0 + u_1, v_0 + v_1) - \fra(u_0, v_0 + v_1)  \\
& = & \fra(u^+,u^-) - \fra(u_0, v_0)
= - \fra(u_0, v_0)  \\
& = & - \fra(u_0, u_0) =  - \fra^D(u_0, u_0).
\end{eqnarray*}
Here we use the fact that 
\begin{align*}
\fra(u^+,u^-)
=& \sum_{k,j=1}^d \int_\Omega a_{kj} \p_k(u^+) \p_j(u^-)  + \sum_{k=1}^d \int_\Omega  \, a_k \p_k u^+ \, u^-   + a_k u^+ \p_k u^- \\
&+  \int_\Omega  \, a_0 u^+ \, u^- = 0.
\end{align*}
By assumption \eqref{2acc} we have $\fra^D(u_0, u_0) \ge 0$ and we obtain 
$\frb (\varphi^+, \varphi^-)  \le 0$.
This proves the positivity of $(T^{\Gamma_1}_t)_{t\ge 0}$   on $L^2(\p\Omega)$.\\

b) By \cite{Ouh96} or \cite{Ouh05}, Corollary 2.17 it suffices to prove that 
$ {\mathbbm 1} \wedge  \varphi := \inf({\mathbbm 1}, \varphi)  \in D(\frb)$ and $\frb({\mathbbm 1} \wedge \varphi, (\varphi - {\mathbbm 1})^+) \ge 0$
for all $\varphi \in D(\frb)$ with $\varphi \geq 0$.
Let $\varphi \in D(\frb)$ and suppose that $\varphi \geq 0$.
Let $u \in H_V$ be real-valued such that $\varphi = \text{Tr} (u)$.  
Note that  ${\mathbbm 1} \wedge u \in V$ (see \cite{Ouh05}, Section 4.3).  
We decompose ${\mathbbm 1} \wedge u = u_0 + u_1 \in W_0^{1,2}(\Omega) \oplus H_V$.
Then
\[
 (u-{\mathbbm 1})^+ = u - {\mathbbm 1} \wedge u = (-u_0) + (u-u_1) \in W_0^{1,2}(\Omega) \oplus H_V.\]
Therefore,
\begin{eqnarray*}
\frb({\mathbbm 1}\wedge \varphi, (\varphi - {\mathbbm 1})^+ ) 
& = & \fra(u_1, u - u_1)
= \fra(u_0 + u_1, u - u_1)  \\
& = & \fra(u_0 + u_1, -u_0 + u - u_1)
   + \fra(u_0 + u_1, u_0)  \\
& = & \fra(u_0 + u_1, -u_0 + u - u_1)
   + \fra(u_0, u_0)  \\
& = &\sum_{k,j=1}^d \int_\Omega a_{kj}\p_k({\mathbbm 1} \wedge u) \p_j((u-{\mathbbm 1})^+) +  \\
&&  \int_\Omega  a_0 ({\mathbbm 1} \wedge u) (u-{\mathbbm 1})^+ + \fra^D(u_0,u_0)\\
& = & \int_\Omega a_0 \, (u-{\mathbbm 1})^+  + \fra^D(u_0,u_0)  
\geq 0.
\end{eqnarray*}
This proves that $\frb({\mathbbm 1} \wedge \varphi, (\varphi - {\mathbbm 1})^+) \ge 0$. 
\end{proof}

 Next we have the following domination property.
\begin{theorem}\label{thm2-2}
Suppose that $a_{kj}$, $a_k$, $\tilde{a_k}$ and $a_0$  satisfy \eqref{sym}. Suppose also that $L^D_a$ is accretive 
 with $0 \notin \sigma(L^D_a)$. Let $\Gamma_0$ and $\tilde{\Gamma_{0}}$ be two closed subsets of the boundary such that $\Gamma_0 \subseteq \tilde{\Gamma_{0}}$. Then for every $0 \le \varphi \in L^2(\p\Omega)$
$$ 0 \le T^{\tilde{\Gamma_{1}}}_t \varphi  \le T^{\Gamma_1}_t \varphi.$$
\end{theorem}
\begin{proof} Let $\tilde{\Gamma_1}$ be the complement of $\tilde{\Gamma_0}$ in $\p\Omega$. Denote by $\frb$ and $\tilde{\frb}$ the sesquilinear forms associated with $\Ng$ and ${\mathcal{N}_{\tilde{\Gamma_1}}}$, respectively. 
Clearly, $\tilde{\frb}$ is a restriction of $\frb$ and hence it 
is enough to prove that $D(\tilde{\frb})$ is an ideal of $D(\frb)$ and apply \cite{Ouh96} or \cite{Ouh05}, Theorem 2.24. For this, let $0 \le \varphi \le \psi$ with $\varphi \in D(\frb)$ and $\psi \in D(\tilde{\frb})$. This means that $\varphi$ and $\psi$ are respectively the traces on $\p\Omega$ of $u, v \in W^{1,2}(\Omega)$ such that 
$$ \varphi =  \text{Tr}(u) = 0 \, \, \text{on } \Gamma_0 \quad \text{and } \psi = \text{Tr}(v) = 0 \, \, \text{on } \tilde{\Gamma_0}.$$
Since $0 \le \varphi \le \psi$  we have $\varphi = 0$ on $\tilde{\Gamma_0}$. This equality gives $\varphi \in D(\tilde{\frb})$ and this shows  that 
$D(\tilde{\frb})$ is an ideal of $D(\frb)$. 
\end{proof}
The next result shows monotonicity with respect to the potential $a_0$. This was already proved in \cite{EO13} Theorem 2.4, in the case where $L_a^D = -\Delta + a_0$. The proof given there works also in the general framework of the present  paper.   

As above let $a_{kj}$, $a_k$ and $a_0$ be real-valued and let $(T^{\Gamma_1, a_0}_t)_{t\ge 0}$ denote the semigroup 
$(T^{\Gamma_1}_t)_{t\ge 0}$ defined above. Suppose that $b_0 $ is a real-valued function and denote by $(T^{\Gamma_1, b_0}_t)_{t\ge 0}$ be the semigroup of
$\Ng$ with coefficients $a_{kj}$, $a_k$ and $b_0$ (i.e. $a_0$ is replaced by $b_0$). 
Then  we have
\begin{theorem}\label{thm2-3}
Suppose that $a_{kj}$, $a_k$, $\tilde{a_k}$ and $a_0$  satisfy \eqref{sym}. Suppose again that $0 \notin \sigma(L^D_a)$ and $L^D_a$ is accretive. If $a_0 \le b_0$ then 
$$ 0 \le T^{\Gamma_1, b_0}_t \varphi \le T^{\Gamma_1, a_0}_t \varphi$$
for all $0 \le \varphi \in L^2(\p\Omega)$ and $t \ge 0$.
\end{theorem}

\section{Proof of  the main result }\label{sec3}
In this section we prove Theorem \ref{thm0}.  We recall briefly the operators introduced  in Section \ref{sec1}. 

For  $\mu \in \R$ and  recall the operator $L_a^\mu$ 
associated with the form $\fra^\mu$  given by \eqref{Rob} with domain $D(\fra^\mu) := V$ and $V$ is again given by \eqref{1-3}.  The operator associated with $\fra^\mu$ is $L_a^\mu$. It 
 is given by  the formal expression \eqref{1-5} and it is subject to mixed  and Robin boundary conditions \eqref{3-6}.\\
 We also recall that  $L_a^D$ is the operator subject to the Dirichlet boundary conditions and $L_a^M$ is subject to mixed boundary conditions. 
 
Fix  $\lambda \notin \sigma(L_a^D)$. We denote by 
${\mathcal{N}_{\Gamma_1, a}}(\lambda)$  the  partial  D-t-N operator with the coefficients $\{ a_{kj}, a_k, a_0 - \lambda \}$. It is the operator associated with the form
$$\frb(\varphi,\psi) :=  \sum_{k,j}^d \int_\Omega a_{kj} \p_k u \overline{\p_j v}\ dx + \sum_{k=1}^d \int_\Omega a_k \p_k u \overline{ v} 
+ \overline{a_k} u  \overline{ \p_k v}\ dx + (a_0 -\lambda) u \overline{v}\ dx$$
where  $u, v \in V_H(\lambda)$ with $\text{Tr}(u) = \varphi, \ \text{Tr}(v) = \psi$ and  
\begin{equation}\label{VHla}
V_H(\lambda)  := \{ u \in V, \fra(u, g) = \lambda \int_\Omega u \overline{g} \,\, \, \text{for all } g \in W^{1,2}_0(\Omega) \},
\end{equation}
This space is the same as in \eqref{VH} but now with $a_0$ replaced by $a_0 - \lambda$. 

We restate the main  theorem using the notation introduced  in Section \ref{sec1}.

\begin{theorem}\label{thm003} Suppose that $\Omega$ is a bounded Lipchitz domain of $\R^d$ with $d \ge 2$. Let $\Gamma_0$ be a closed  subset of $\p\Omega$, $\Gamma_0 \not= \p\Omega$  and $\Gamma_1 = \p \Omega \setminus \Gamma_0$. Let $a= \{ a_{kj} = \overline{a_{jk}}, a_k = \overline{\tilde{a_k}} , a_0 = \overline{a_0} \}$ and $b = \{ b_{kj} = \overline{b_{jk}}, b_k = \overline{\tilde{b_k}}, b_0 = \overline{b_0} \}$ be bounded measurable functions on $\Omega$ such that $a_{kj}$ and $b_{kj}$ satisfy the ellipticity condition \eqref{1-1}. If $d \ge 3$ we assume in addition that  the coefficients $a_{kj}, b_{kj}, a_k$ and $b_k$ are Lipschitz continuous on  $\overline{\Omega}$. \\
Suppose that ${\mathcal{N}_{\Gamma_1, a}}(\lambda)  = {\mathcal{N}_{\Gamma_1, b}}(\lambda)$  for all  $\lambda $ in a set having an accumulation point in $\rho(L^D_a) \cap \rho(L^D_b)$. Then:\\
$i)$ The operators $L_a^\mu$ and $L_b^\mu$ are unitarily equivalent for all $\mu \in \R$.\\
$ii)$  The operators $L_a^M$ and $L_b^M$ are unitarily equivalent.\\
$iii)$ The operators $L_a^D$ and $L_b^D$  are unitarily equivalent.

Moreover, for every   $\lambda  \in \sigma(L_a^\mu) = \sigma(L_b^\mu)$ with $\lambda \notin \sigma(L_a^D) = \sigma(L_b^D)$,  the sets
$\{ \text{Tr}(u),  u  \in \text{Ker}(\lambda I - L_a^\mu) \} $  and $\{ \text{Tr}(v),  v  \in \text{Ker}(\lambda I - L_b^\mu) \} $ coincide. 
 The same property holds for the operators $L_a^M$ and $L_b^M$.
\end{theorem}

We shall need several preparatory results. We start with the following theorem  which was proved in \cite{AM07} and \cite{AM12} in the case where 
$a_{kj} = \delta_{kj}$, $a_k = 0$,  $a_0$ is a constant and $\Gamma_1 = \p\Omega$.  
\begin{theorem}\label{thm3-1} Let  $a= \{ a_{kj} = \overline{a_{jk}}, a_k = \overline{\tilde{a_k}} , a_0 = \overline{a_0} \}$ 
be  bounded measurable functions on $\Omega$ such that $a_{kj}$ satisfy the ellipticity condition \eqref{1-1}.\\
Let $\mu, \lambda \in \R$ and  $\lambda \notin \sigma(L^D_a)$. Then:\\
1) $\mu \in \sigma({\mathcal{N}_{\Gamma_1, a}}(\lambda))  \Leftrightarrow \lambda \in \sigma(L_a^\mu)$. In addition, if $u \in \text{Ker}(\lambda - L_a^\mu)$, $u \not= 0$ then $\varphi := \text{Tr}(u) \in \text{Ker}(\mu - {\mathcal{N}_{\Gamma_1, a}}(\lambda))$ and $\varphi \not = 0$. Conversely, if $\varphi \in \text{Ker}(\mu - {\mathcal{N}_{\Gamma_1, a}}(\lambda))$, $\varphi \not= 0$, then there exists $u \in \text{Ker}(\lambda - L_a^\mu)$, $u \not= 0$ such that $\varphi = \text{Tr}(u)$.\\
2) $\text{dim Ker} (\mu - {\mathcal{N}_{\Gamma_1, a}}(\lambda)) = \text{dim Ker} (\lambda- L_a^\mu)$.
\end{theorem}
\begin{proof}  We follow  a similar  idea  as in   \cite{AM07} and  \cite{AM12}.  It is enough to prove that the mapping
$$ S :  \text{Ker}(\lambda - L_a^\mu) \to  \text{Ker}(\mu - {\mathcal{N}_{\Gamma_1, a}}(\lambda)), \ u \mapsto \text{Tr}(u)$$
is an isomorphism. 
First, we prove that $S$ is well defined. Let $u \in \text{Ker}(\lambda - L_a^\mu)$. Then $u \in D(L_a^\mu)$ and  $L_a^\mu u = \lambda u$. By the definition of $L_a^\mu$ we have $u \in V$ and for all $v \in V$
\begin{eqnarray}
&& \sum_{k,j=1}^d \int_\Omega a_{kj} \p_k u \overline{\p_j v} + \sum_{k=1}^d \int_\Omega a_k \p_k u \overline{v} + \overline{a_k} u \overline{\p_k v} \nonumber\\
&& + \int_\Omega a_0 u \overline{v} - \lambda \int_\Omega u \overline{v} = \mu \int_{\p\Omega} \text{Tr}(u) \overline{\text{Tr}(v)}. \label{3-3}
\end{eqnarray}
Taking $v \in W^{1,2}_0(\Omega)$ yields $u \in V_H(\lambda)$. Note that \eqref{3-3} also holds for $v \in V_H(\lambda)$. Hence
 it follows from the  definition of ${\mathcal{N}_{\Gamma_1, a}}(\lambda)$ that 
 $$ \varphi := \text{Tr}(u) \in D({\mathcal{N}_{\Gamma_1, a}}(\lambda)) \, \, \text{and } {\mathcal{N}_{\Gamma_1, a}}(\lambda) \varphi  = \mu \varphi.$$
 This means that $S(u) \in \text{Ker}(\mu - {\mathcal{N}_{\Gamma_1, a}}(\lambda))$. 
 
 Suppose now that $u \in \text{Ker}(\lambda - L_a^\mu)$ with $u \not= 0$. If $S(u) = 0$ then $u \in W^{1,2}_0(\Omega)$. Therefore, it follows from \eqref{3-3} that for all $v \in V$
 \begin{equation}\label{3-4}
  \sum_{k,j=1}^d \int_\Omega a_{kj} \p_k u \overline{\p_j v} + \sum_{k=1}^d \int_\Omega a_k \p_k u \overline{v} + \overline{a_k} u \overline{\p_k v} 
  + \int_\Omega (a_0 - \lambda)  u \overline{v} = 0.
  \end{equation}
  This implies  that $u \in V_H(\lambda)$. We conclude by Lemma \ref{lem1} that $u = 0$. Thus $S$ is injective. \\
  We prove that $S $ is surjective. Let $\varphi \in \text{Ker}(\mu - {\mathcal{N}_{\Gamma_1, a}}(\lambda))$. Then by the definition of 
  $\mathcal{N}_{\Gamma_1, a}(\lambda)$, there exists $u \in V_H(\lambda)$ such that 
  $\varphi = \text{Tr}(u)$ and $u$ satisfies \eqref{3-3} for all $v \in V_H(\lambda)$. If $v \in V$ we write 
  $v = v_0 + v_1 \in W^{1,2}_0(\Omega) \oplus V_H(\lambda)$ and see that \eqref{3-3} holds for $u$ and $v$. This means that $u \in D(L_a^\mu)$ and $L_a^\mu u = \lambda u$. 
\end{proof}
\begin{lemma}\label{lem5} For $\lambda \in \R$ large enough, $(\lambda + L_a^\mu)^{-1}$ converges in ${\mathcal L}(L^2(\Omega))$ to $(\lambda + L_a^D)^{-1}$ as $\mu \to -\infty$.
\end{lemma} 
This is Proposition 2.6 in \cite{AM07} when $a_{kj} = \delta_{kj},\ a_k = a_0 = 0$. The proof given in \cite{AM07} remains valid in  our setting. Note that the idea of proving the uniform convergence  here is based on a criterion from \cite{Dan03} (see Appendix B) which states that it is enough  to check that for all $(f_n), f \in L^2(\Omega)$
\begin{equation}\label{1-1-1}
f_n  \rightharpoonup f  \Rightarrow (\lambda + L_a^{\mu_n})^{-1} f_n \to (\lambda + L_a^D)^{-1} f,
\end{equation}
for every sequence $\mu_n \to -\infty$. The first convergence is in the weak sense  in $L^2(\Omega)$ and the second   one is the strong convergence.  It is not difficult to check \eqref{1-1-1}. 

\vspace{.3cm}

From now on, we denote by $(\lambda_{a,n}^\mu)_{n\ge1}$ the eigenvalues of $L_a^\mu$, repeated according to their multiplicities. We have for each $\mu \in \R$
$$\lambda_{a,1}^\mu \le \lambda_{a,2}^\mu\le \dots  \to + \infty.$$
Similarly for the eigenvalues $(\lambda_{a,n}^D)_{n\ge1}$ of $L_a^D$. These eigenvalues satisfy the standard min-max principle since the operators $L_a^\mu$ and $L_a^D$ are self-adjoint by our assumptions. 
 
 A well known  consequence of the previous lemma is that the spectrum of $L_a^\mu$ converges to the spectrum of $L_a^D$. More precisely, for all $k$, 
\begin{equation}\label{sigmaconv}
\lambda_{a,k}^\mu \to \lambda_{a,k}^D \text{ as}\  \mu \to -\infty.
\end{equation}
In addition, we have the following lemma which will play a fundamental role. 
\begin{lemma}\label{decrease} Let $a= \{ a_{kj} = \overline{a_{jk}}, a_k = \overline{\tilde{a_k}} , a_0 = \overline{a_0} \}$ 
 be bounded measurable functions on $\Omega$ such that $a_{kj}$  satisfy the ellipticity condition \eqref{1-1}. If $d \ge 3$ we assume in addition that  the coefficients $a_{kj}$ and $a_k$  are Lipschitz continuous on  $\overline{\Omega}$. 
Then for each $k$, $\mu \mapsto \lambda_{a,k}^\mu$ is strictly decreasing on $\R$ and $\lambda_{a,k} \to -\infty$ as $\mu \to + \infty$. 
\end{lemma}
\begin{proof} Firstly, by the min-max principle $ \lambda_{a,k}^\mu \le  \lambda_{a,k}^D$ 
and the function $\mu \mapsto \lambda_{a,k}^\mu$ is non-increasing. Fix $ k\ge 0$ and suppose that $\mu \mapsto \lambda_{a,k}^\mu$ is constant on $[\alpha,\beta]$ for some $\alpha < \beta$. 
For each $\mu$ we take a normalized eigenvector $u^\mu$ such that $\text{Tr}(u^{\mu + h}) \to \text{Tr}(u^{\mu})$ in $L^2(\partial\Omega)$ as
$h \to 0$ (or as $h_n \to 0$ for some sequence $h_n$).  Indeed,  due to regularity properties $\mu \mapsto \lambda_{a, k}^{\mu}$ is continuous (see \cite{Kat1}, Chapter VII) and hence
$(\lambda_{a, k}^{\mu +h})_h$ is bounded for small $h$. The equality  $\fra^{\mu+h}(u^{\mu +h}, u^{\mu +h}) = \lambda_{a, k}^{\mu +h}$ implies that 
$\fra^{\mu+h}(u^{\mu +h}, u^{\mu +h})$ is bounded w.r.t. $h$ (for small $h$). This latter property and  ellipticity easily imply that $(u^{\mu +h})_h$ is bounded in $V$. After extracting a sequence we may assume that $(u^{\mu +h})_h$ converges weakly in $V$ to some $u$ as $h \to 0$. The compactness embedding of $V$ in $L^2(\Omega)$ as well as the compactness of the trace operator show that $(u^{\mu +h})_h$ converges 
to $u$ in $L^2(\Omega)$ and $\text{Tr}(u^{\mu + h})$ converges to $ \text{Tr}(u)$ in $L^2(\partial\Omega)$. On the other hand for every $v \in V$, the equality 
$$ \fra^{\mu+h}(u^{\mu +h}, v) = \lambda_{a, k}^{\mu +h} \int_\Omega u^{\mu +h} v \ dx$$ 
 shows that the limit $u$ is a normalized  eigenvector of $L_a^\mu$ for the eigenvalue $\lambda_{a,k}^\mu$. We take $u^\mu := u$ and obtain the claim stated above. 

Observe that 
\begin{equation}\label{ortho}
\int_{\Gamma_1} \text{Tr}(u^{\mu + h}) \overline{\text{Tr}(u^\mu)} \ d\sigma = 0
\end{equation}
for all $h \not= 0$ and $\mu, \mu + h \in [\alpha, \beta]$. Indeed, using the definition of the  form $\fra^\mu$ (see \eqref{Rob}) we have
\begin{eqnarray*}
\lambda \int_\Omega u^{\mu + h} \overline{u^{\mu}}\ dx &=& \fra^{\mu +h}(u^{\mu + h}, u^\mu)\\
&=& \fra^{\mu}(u^{\mu + h}, u^\mu) - h \int_{\Gamma_1} \text{Tr}(u^{\mu + h}) \overline{\text{Tr}(u^{\mu})} \ d\sigma\\
&=& \lambda \int_\Omega u^{\mu + h} \overline{u^{\mu}}\ dx - h \int_{\Gamma_1} \text{Tr}(u^{\mu + h}) \overline{\text{Tr}(u^{\mu})}\ d\sigma.
\end{eqnarray*}
This gives \eqref{ortho}. Now, letting $h \to 0$ we obtain from \eqref{ortho} and the fact that  $ \text{Tr}(u^{\mu + h})$ converges to $ \text{Tr}(u^{\mu })$ 
as $h \to 0$ that $ \text{Tr}(u^{\mu }) = 0$ on $\Gamma_1$ for all $\mu \in [\alpha, \beta]$.  Hence $ \text{Tr}(u^{\mu }) = 0$  on $\partial \Omega$ 
since $u^\mu \in V$. Hence, $L^\mu$ has an eigenfunction  $u^\mu \in W^{1,2}_0(\Omega)$. 
Note that if $d = 2$ or if $d \ge 3$ and the coefficients $a_{kj}$ and $a_k$ are Lipschitz continuous on $\overline{\Omega}$, 
then the operator $L_a$ has the unique continuation property (see \cite{Sc} for the case $d= 2$ and \cite{Wo} 
for  $d \ge 3$). If $d \ge 3$ and hence the coefficients are Lipschitz on $\overline{\Omega}$, we apply  
Proposition 2.5 in \cite{BR12} to conclude  that  $u^\mu = 0$, but this is not possible since $\|u^\mu\|_2 = 1$.  If 
$d = 2$ we argue in a similar way. Indeed, let $\widetilde{\Omega} $ be an open subset of $\R^2$ containing $\Omega$ and such that $\Gamma_0 \subset \partial \widetilde{\Omega}$ and $\widetilde{\Omega}\setminus \Omega$ contains an open ball. We extend all the coefficients to bounded measurable function $\tilde{a}_{kj}, \tilde{a}_k$ and $\tilde{a}_0$ on $\widetilde{\Omega}$. In addition,  $\tilde{a}_{kj} = \overline{\tilde{a}_{jk}}$  on $ \widetilde{\Omega}$ and satisfy the ellipticity condition. We extend $u^\mu$ to 
$\tilde{u^\mu} \in W^{1,2}_0( \widetilde{\Omega})$ by $0$ outside $\Omega$.  We define in  $\widetilde{\Omega}$ the elliptic 
operator $L_{\tilde{a}}$ as previously. For $v \in C_c^\infty( \widetilde{\Omega})$ we note that $v_{\vert \Omega} \in V$ and hence
\begin{eqnarray*}
\int_{\widetilde{\Omega}} L_{\tilde{a}} (\tilde{u^\mu})  \overline{v}  dx &=& \fra^\mu( u^\mu, v_{\vert \Omega})\\
 &=& \lambda \int_\Omega u^\mu  \overline{v}_{\vert \Omega} =  \lambda \int_{ \widetilde{\Omega}}  \tilde{u^\mu}   \overline{v}.
\end{eqnarray*}
The  term  $\int_{\widetilde{\Omega}} L_{\tilde{a}} (\tilde{u^\mu})  \overline{v}$ is of course interpreted in the sense of the associated 
sesquilinear form and the first equality uses the fact that $\tilde{u^\mu}$ is $0$ 
on $ \widetilde{\Omega} \setminus \Omega$ and $u^\mu \in W^{1,2}_0(\Omega)$.  Hence, $\tilde{u^\mu}$ satisfies
$$ (L_{\tilde{a}} - \lambda) (\tilde{u^\mu}) = 0$$
in the weak sense on $ \widetilde{\Omega}$. We conclude by the unique continuation property  (\cite{Sc}) that $\tilde{u^\mu} = 0$ on 
$\widetilde{\Omega}$ since it is $0$  on an open ball contained in $ \widetilde{\Omega} \setminus \Omega$ . We arrive  as above to a contradiction. 
Hence, $\mu \mapsto \lambda_{a,k}^\mu$ is strictly decreasing on $\R$.

It remains to prove that for any $k$, $\lambda_{a,k}^\mu \to -\infty$ as $\mu \to +\infty$. By the min-max principle
$$\lambda_1^\mu \le \sum_{k,j=1}^d \int_{\Omega} a_{kj} \p_k u \overline{\p_j u} + 2 \Re \sum_{k=1}^d \int_\Omega a_k \p_k u \overline{u} + 
\int_{\Omega} a_0 | u |^2 - \mu \int_{\Gamma_1} | \text{Tr}(u)|^2 $$
for every normalized $u \in V$. Taking $u$ such that $\text{Tr}(u) \not= 0$ shows that $\lambda_{a,1}^\mu \to -\infty$ as 
$\mu \to +\infty$. Suppose now that $\lambda_{a, k}^\mu  > w $ for some $w \in \R$,  $ k > 1$ and all $\mu \in \R$.  
Taking the smallest possible $k$ we have $\lambda_{a,j}^\mu \to -\infty$ as $\mu \to +\infty$ for $j= 1, \cdots, k-1$. Of course,
$\lambda_{a, j}^\mu  > w $ for all $j \ge k$ and we may choose $w \notin \sigma(L_a^D)$. Remember also that  
$\mu \mapsto \lambda_{a, j}^\mu$ is strictly decreasing for $j= 1, \cdots, k-1$. 
On the other hand, by Theorem \ref{thm3-1} we have 
$\sigma({\mathcal{N}_{\Gamma_1, a}}(w)) \subset \{ \mu \in \R, \lambda_{a,j}^\mu = w, j = 1, \cdots, k-1\}$. Using the fact that
$\lambda_{a,j}^\mu \to -\infty$ as $\mu \to +\infty$  and $\mu \mapsto \lambda_{a, j}^\mu$ is strictly decreasing for $j = 1, \cdots, k-1$  we see that we  can choose $w$ such that the set $\{ \mu \in \R, \lambda_{a,j}^\mu = w, j = 1, \cdots, k-1\}$ is finite and hence $\sigma({\mathcal{N}_{\Gamma_1, a}}(w))$ is finite which is not possible since $L^2(\Gamma_1)$ has infinite dimension. 
\end{proof}

Related results to Lemma \ref{decrease} can be found in \cite{AM12} (see Proposition 3) and  \cite{R14}. In both papers  the proofs use the unique continuation property.   
 
We shall also need the following lemma.
\begin{lemma}\label{lem4}
For every $\varphi, \psi \in \text{Tr}(V)$, the mapping
$$\lambda \mapsto  \langle \mathcal{N}_{\Gamma_1, a}(\lambda)  \varphi, \psi \rangle$$
is holomorphic  on $\C \setminus \sigma(L_a^D)$.
\end{lemma}
This result  is easy to prove, see Lemma 2.4 in \cite{BR12}.

\begin{proof}[Proof of Theorem \ref{thm0}]  
As above, we denote by $(\lambda_{b,n}^\mu)_{n\ge1}$ and $(\lambda_{b,n}^D)_{n\ge1}$ the eigenvalues of the self-adjoint operators 
$L_b^\mu$ and $L_b^D$, respectively.

It follows from Lemma \ref{lem4} and the assumptions that  $\mathcal{N}_{\Gamma_1, a}(\lambda) = 
 \mathcal{N}_{\Gamma_1, b}(\lambda)$ for all $\lambda \in \C \setminus (\sigma(L_a^D) \cup \sigma(L_b^D))$. \\
 
 $i)$ We show that for all $\mu \in \R$
 \begin{equation}\label{3-8}
 \sigma(L_a^\mu) = \sigma(L_b^\mu), 
 \end{equation}
 and the eigenvalues have the same multiplicity. \\
 Fix $\mu \in \R$ and suppose that  $\lambda = \lambda_{a,k}^\mu \in \sigma(L_a^\mu) \setminus (\sigma(L_a^D) \cup \sigma(L_b^D))$. 
 By Theorem \ref{thm3-1}, 
 $ \mu \in \sigma(\mathcal{N}_{\Gamma_1, a}(\lambda) )= \sigma(
 \mathcal{N}_{\Gamma_1, b}(\lambda)) $
 and hence $\lambda \in \sigma(L_b^\mu)$. Thus, $\lambda = \lambda_{a,k}^\mu = \lambda_{b,j}^\mu$ for some $j \ge 1$. 
 The second  assertion of Theorem \ref{thm3-1}  shows that $\lambda_{a,k}^\mu $ and $ \lambda_{b,j}^\mu$ have the same multiplicity. In addition,  $j= k$. Indeed, if $k < j$ then  
 $$\lambda_{b,1}^\mu \le  \lambda_{b,2}^\mu \le  \dots \le \lambda_{b,k}^\mu \le  \dots \le \lambda_{b,j}^\mu = \lambda_{a,k}^\mu.$$
Each $\lambda_{b,m}^\mu$ coincides with an eigenvalue of $L_a^\mu$ (with the same multiplicity) and hence  $\lambda_{a,k}^\mu$ is (at least) the $j-$th eigenvalue of $L_a^\mu$ with $j > k$ which is not possible. The same argument works if $j < k$. 
Using  Lemma \ref{decrease} we see that  for any  $k$ there exists a discrete set $J \subset \R$ such that $ \lambda_{a,k}^\mu = \lambda_{b,k}^\mu$ for every $\mu \in \R \setminus J$.    By continuity of $\mu \mapsto \lambda_{a,k}^\mu$ and $\mu \mapsto \lambda_{b,k}^\mu$ these two functions coincide 
 on $\R$. This proves  \eqref{3-8} and also that the multiplicities of the  eigenvalues $\lambda_{a,k}^\mu$ and $\lambda_{b,k}^\mu$ are the same.  

The similarity property follows by  a classical argument. Recall that  $L_a^\mu$ and $ L_b^\mu$  are self-adjoint operators with  compact resolvents. It follows  that here exist 
orthonormal bases $\Phi_n$ and $\Psi_n$  of $L^2(\Omega)$ which are eigenfunctions of $L_a^\mu$ and $L_b^\mu$, respectively. 
Define the mapping
$$\mathcal U : L^2(\Omega) \to L^2(\Omega), \, \Phi_n \mapsto \Psi_n.$$
Thus for  $ f = \sum_n (f, \Phi_n) \Phi_n \in L^2(\Omega)$, ${\mathcal U}(f) = \sum_n (f, \Phi_n) \Psi_n$. The notation $(f, \Phi_n)$ is the scalar product in $L^2(\Omega)$. 
Clearly, 
$$ \| {\mathcal U} (f) \|_2^2 = \sum_n | (f, \Phi_n)|^2 = \| f \|_2^2.$$
The mapping $\mathcal U$ is  an isomorphism. In addition, if $L_a^\mu \Phi_n = \lambda_{a,n}^\mu  \Phi_n$ then for $f \in D(L_b^\mu)$ 
\begin{eqnarray*}
{\mathcal U}  L_a^\mu {\mathcal U}^{-1} (f) &=& {\mathcal U} L_a^\mu {\mathcal U}^{-1}\left( \sum_n (f, \Psi_n) \Psi_n \right)\\
&=& {\mathcal U} L_a^\mu\left( \sum_n (f, \Psi_n) \Phi_n \right)\\
&=& {\mathcal U} \left( \sum_n (f, \Psi_n) \lambda_{a,n}^\mu  \Phi_n \right)\\
&=& \sum_n (f, \lambda_{b,n}^\mu \Psi_n)  \Psi_n\\
&=& L_b^\mu(f).
\end{eqnarray*}
Thus, $L_a^\mu$ and $L_b^\mu$ are unitarily equivalent. This proves assertion $i)$.\\

$ii)$ Choose $\mu = 0$ in the previous assertion.\\

$iii)$ As mentioned above, by  Lemma \ref{lem5} we have \eqref{sigmaconv}. The same property  holds for  $L_b^\mu$, that is, 
 $\lambda_{b,k}^\mu \to \lambda_{b,k}^D$ as $\mu \to -\infty$. It follows from assertion $(i)$ that $\lambda_{a,k}^D = \lambda_{b,k}^D$ for all
 $k \ge 1$ and have the same multiplicity.  We conclude as above that $L_a^D$ and $L_b^D$ are unitarily equivalent. \\

Finally, another  application of Theorem \ref{thm3-1} shows that  $\text{Tr}(\text{Ker} (\lambda- L_a^\mu)) = \text{Tr}(\text{Ker} (\lambda- L_b^\mu))$ for $\lambda \notin \sigma(L_a^D)= \sigma(L_b^D)$. \end{proof}

\noindent{\bf Acknowledgements.} The author wishes to thank  Jonathan Rohleder for  very helpful remarks and discussions  on the proof of  Lemma \ref{decrease}. He wishes also to thank the referee for his/her very 
careful reading of the manuscript.

\noindent \emph{El Maati Ouhabaz,} Institut de Math\'ematiques (IMB), Univ.\  Bordeaux, 351, cours de la Libération, 33405 Talence cedex, France,\\ 
\texttt{Elmaati.Ouhabaz@math.u-bordeaux1.fr}\\

\noindent The research of the author was partially supported by the ANR
project HAB, ANR-12-BS01-0013-02.


\begin{thebibliography}{1}
 \bibitem{ABHN} W. Arendt, C.J.K. Batty, M. Hieber and F.  Neubrander, \textit{Vector-Valued Laplace Transforms and Cauchy Problems}. Second edition. Monographs in Mathematics, 96. Birkhäuser/Springer Basel AG, Basel, 2011.
\bibitem{AM07} W. Arendt and R. Mazzeo, Spectral properties of the Dirichlet-to-Neumann operator on Lipschitz domains, \textit{Ulmer Seminare} 2007.
\bibitem{AM12} W. Arendt and R. Mazzeo, Friedlander's eigenvalue inequalities and the Dirichlet-to-Neumann semigroup, 
\textit{ Commun. Pure Appl. Anal.}  11 (2012), no. 6, 2201–2212. 
\bibitem{AP06}  K. Astala and L. P\"aiv\"arinta,   Calder\'on's inverse conductivity problem in the plane, 
 \textit{ Ann. of Math.} (2) 163 (2006), no. 1, 265–299.
\bibitem{ALP}
K. Astala, M. Lassas, and  L. P\"aiv\"arinta,     Calder\'on's inverse problem for anisotropic conductivity in the plane,  \textit{Comm. Partial Differential Equations} 30 (2005), no. 1-3, 207–224.
\bibitem{BR12} J. Behrndt and J. Rohleder,  An inverse problem of Calder\'on type with partial data,\textit{ Comm. Partial Differential Equations}
 37 (2012), no. 6, 1141–1159.
 \bibitem{CR} P. Caro and K. Rogers, Global uniqueness for the Calderón problem with Lipschitz conductivities, 
 http://arxiv.org/abs/1411.8001.
 \bibitem{Dan03} D. Daners, Dirichlet problems on varying domains, \textit{J. Diff. Eqs} 188 (2003) 591-624.
 \bibitem{Dav2} E.B. Davies, \textit{Heat Kernel and Spectral Theory}, Cambridge Tracts in Math. 92, Cambridge Univ. Press 1989.
 \bibitem{DKSU} D. Dos Santos Ferreira, C.E. Kenig, M. Salo and  G. Uhlmann, Limiting
Carleman weights and anisotropic inverse problems,  \textit{Invent. Math. } 178
(2009), 119–171.
 \bibitem{EO13} A.F.M. ter Elst and E.M. Ouhabaz, Analysis of the heat kernel of the Dirichlet-to-Neumann operator, \textit{J. Functional Analysis} 267 (2014) 4066-4109.
 \bibitem{GLU} A. Greenleaf, M. Lassas and G.  Uhlmann, The Calderon problem for conormal potentials, I:
Global uniqueness and reconstruction, \textit{Comm. Pure Appl. Math.}  2003, 56, 328-352.
\bibitem{HT} B. Haberman and D. Tataru, 
Uniqueness in Calder\'on's problem with Lipschitz conductivities, 
\textit {Duke Math. J. } 162 (2013), no. 3, 496–516.
\bibitem{Hab} B. Haberman,  Uniqueness in Calder\'on's problem for conductivities with unbounded gradient, 
http://arxiv.org/abs/1410.2201. 
 \bibitem{Isa} V. Isakov, On uniqueness in the inverse conductivity problem with local
data, \textit{ Inverse Probl. Imaging} 1 (2007), 95–105.
 \bibitem{IUY} O. Imanuvilov, G. Uhlmann and M. Yamamoto, The Calder\'on problem
with partial data in two dimensions, \textit{ J. Amer. Math. Soc. } 23 (2010),
655–691.
\bibitem{Kat1}
T. Kato, \textit{Perturbation Theory for Linear Operators}.
Second edition, Grundlehren der mathematischen Wissenschaften 132.
  Springer-Verlag, Berlin etc., 1980.
 \bibitem{KSU} C.E. Kenig, J. Sj\"ostrand and G. Uhlmann, The Calder\'on problem with
partial data, \textit{Ann. of Math.} 165 (2007), 567–591.
 \bibitem{KS} C.E. Kenig and M. Salo, Recent progress in the Calder\'on problem with partial data, 
 http://arxiv.org/abs/1302.4218. 
 \bibitem{LU89} J. Lee and G. Uhlmann, Determining anisotropic real-analytic
conductivities by boundary measurements, \textit{Comm. Pure Appl. Math.} 42 (1989) 1097–1112.
\bibitem{Na96} A. Nachman, Global uniqueness for a two-dimensional inverse boundary value problem, 
\textit{Ann. of Math.} 1996, 143(2), 71–96.
  \bibitem{Ouh96} E.M. Ouhabaz,  Invariance of closed convex sets and domination criteria for semigroups, \textit{ Potential Anal.} 5 
  (1996), no. 6, 611–625.
\bibitem{Ouh05} E. M. Ouhabaz, \textit{Analysis of Heat Equations on Domains}. London Math. Soc. Monographs,
Princeton Univ.\ Press 2005. 
\bibitem{Ouh04} E.M. Ouhabaz, Gaussian upper bounds for heat 
kernels of second-order elliptic operators with complex coefficients 
on arbitrary domains, \textit{J. Operator Theory} 51  (2004) 335-360.
\bibitem{R14} J. Rohleder, Strict inequality of Robin eigenvalues for elliptic differential operators on Lipschitz domains,
\textit{J. Math. Anal. Appl.} 418 (2014) 978–984.
\bibitem{Sc} F. Schulz, On the unique continuation property of elliptic divergence form
equations in the plane, \textit{Math. Z.}  228 (1998), 201-206.
\bibitem{SU87} J. Sylvester and G. Uhlmann, A global uniqueness theorem for an inverse boundary value problem,
\textit{Ann. of Math.} 1987, 125(1), 153–169.
\bibitem{Wo}  T. H. Wolff,  Recent work on sharp estimates in second-order elliptic unique continuation problems,
\textit{J. Geom. Anal.} 3 (1993), no. 6, 621-650.

\end{thebibliography}
\end{document}